\documentclass[12pt]{article}
\usepackage[utf8]{inputenc}
\usepackage{amsthm}

\usepackage{hyperref, xcolor}
\usepackage{scrextend}
\hypersetup{
    colorlinks=true,
    urlcolor=magenta,
}
\usepackage{academicons}
\definecolor{orcidlogocol}{HTML}{A6CE39}
\usepackage{amsmath,amssymb,amsfonts,oldlfont,graphics,oldgerm,latexsym,dsfont}
\usepackage{graphics}
\usepackage{pstricks,pst-node,pst-arrow}
\usepackage{latexsym,amssymb}
\foreach \x in {A, ..., Z}{%
	\expandafter\xdef\csname orcid\x\endcsname{\noexpand\href{https://orcid.org/\csname orcidauthor\x\endcsname}{\noexpand\orcidicon}}
}

\usepackage{enumerate}
\usepackage{enumitem}
\usepackage{todonotes}

\PassOptionsToPackage{usenames,dvipsnames,svgnames}{xcolor}  
\usepackage{tikz}
\usetikzlibrary{arrows,positioning,automata, fit}
\usetikzlibrary{arrows.meta}
\usepackage{float}

\newtheorem{theorem}{Theorem}
\newtheorem{lemma}[theorem]{Lemma}
\newtheorem{corollary}[theorem]{Corollary}
\newtheorem{proposition}[theorem]{Proposition}

\newtheorem{conjecture}[theorem]{Conjecture}
\usepackage{mathtools}
\DeclarePairedDelimiter{\ceil}{\lceil}{\rceil}
\DeclarePairedDelimiter\floor{\lfloor}{\rfloor}
\usepackage{amssymb}

\newcommand{\dist}{{\mathrm{dist}}}
\newcommand{\Aut}{{\mathrm{Aut}}}

\usepackage{authblk}

\providecommand{\keywords}
{
  \small	
  \noindent \textbf{Keywords:} automorphism; symmetry breaking; majority edge coloring; majority distinguishing index;
}

\providecommand{\msc}
{
  \small	
  \noindent \textbf{Mathematics Subject Classification:} 05C15, 05C20, 05C25
}

\title{Majority distinguishing edge coloring}
\author{Aleksandra Gorzkowska, Magdalena Prorok} 
\affil{\small AGH University of Krakow \\ \small al. Mickiewicza 30, 30-059 Krak\'{o}w, Poland\\
\small\tt {agorzkow, prorok}@agh.edu.pl}

\begin{document}

\maketitle
\begin{abstract}
We consider edge colorings of graphs. An edge coloring is a majority coloring if for every vertex at most half of the edges incident with it are in one color. And edge coloring is a distinguishing coloring if for every non-trivial automorphism at least one edge changes its color. We consider these two notions together. We show that every graph without pendant edges has a majority distinguishing edge coloring with at most $\ceil{\sqrt{\Delta}}+5$ colors. Moreover, we show results for some classes of graphs and a~general result for symmetric digraphs.
\end{abstract}

\keywords

\msc

\section{Introduction}\label{sec:introduction}

Bock et al. in~\cite{BKPPRW} introduced majority edge coloring of graphs: For a (finite, simple, and undirected) graph $G$, an edge coloring $c : E(G) \rightarrow [k]$ is a majority edge $k$-coloring if, for every vertex $u$ of $G$ and every color $\alpha$ in $[k]$, at most half of the edges incident with $u$ have the color $\alpha$. They were interested in finding the smallest number of colors in a~majority edge coloring of a graph. 

They proved the following theorems:

\begin{theorem}~\cite{BKPPRW}\label{thm:2colors}
    Let $G$ be a connected graph. 
    \begin{enumerate}[label=(\roman*)]
        \item If $G$ has an even number of edges or $G$ contains vertices of odd degree, then $G$ has an edge 2-coloring such that, for every vertex $u$ of $G$, at most $\ceil{\frac{d_{G}(u)}{2}}$ of the edges incident with u have the same color. 
        \item If $G$ has an odd number of edges, all vertices of $G$ have even degree, and $u_G$ is any vertex of $G$, then $G$ has an edge 2-coloring such that, for every vertex $u$ of $G$ distinct from $u_G$, exactly $\frac{d_G(u)}{2}$ of the edges incident with $u$ have the same color, and exactly $\frac{d_G(u_G)}{2} + 1$ of the edges incident with $u_G$ have the same color.
    \end{enumerate}
\end{theorem}

In this paper whenever we use the coloring from Theorem~\ref{thm:2colors}, we refer to the vertex $u_G$ as a \textit{special vertex}. 

\begin{theorem}~\cite{BKPPRW}\label{thm:deg2}
    Every graph of minimum degree at least 2 has a majority edge 4-coloring.
\end{theorem}

\begin{theorem}~\cite{BKPPRW}\label{thm:deg4}
    Every graph of minimum degree at least 4 has a majority edge 3-coloring.
\end{theorem}

In their paper they did not associate a parameter with majority colorings. The {\it majority index} $M'(G)$ of a graph $G$ is the smallest integer $k$ such that $G$ has a~majority edge coloring with $k$ colors.  

Distinguishing coloring was introduced in 2015 by Kalinowski and Pil\'sniak in~\cite{KP}. An edge coloring of graph $G$ is {\it distinguishing} if it is preserved only by the identity automorphism of $G$. The smallest number of colors needed for such a~coloring is called {\it the distinguishing index} of $G$ and denoted by $D'(G)$. In their original paper Kalinowski and Pil\'sniak proved the following.

\begin{theorem}\cite{KP}\label{thm:D'}
    If $G$ is a connected graph of order $n \geq 3$, then $D'(G) \leq \Delta(G)$ except for three small cycles $C_3$, $C_4$ or $C_5$.
\end{theorem}

In the same paper they also defined the distinguishing chromatic index $\chi'_D(G)$ of a graph $G$ as the smallest number $d$ such that $G$ has a proper edge coloring with $d$ colors that is preserved only by the identity automorphism of $G$.

\begin{theorem}\cite{KP}\label{thm:KP15}
    If $G$ is a connected graph of order $n \geq 3$, then $$\chi'_D(G) \leq \Delta(G)+1$$ except for four graphs of small order $C_4, K_4, C_6, K_{3,3}$.
\end{theorem}

In this article we consider edge colorings of graphs that are both distinguishing and majority colorings. We introduce the {\it majority distinguishing index} $M'_D(G)$ of a graph~$G$ as the smallest number of colors in such a coloring. We notice that for graphs with $\Delta(G) \leq 3$ the majority coloring is precisely the proper coloring. Therefore, we will not examine this case. For any graph $G$ trivially holds $M'_D(G) \leq \chi'_D(G)$. We also consider this coloring and parameter for symmetric digraphs. For a (finite) directed graph $G$ (called also a digraph), an arc coloring $c : A(G) \rightarrow [k]$ is a \textit{majority arc $k$-coloring} if, for every vertex~$u$ of $G$ and every color $\alpha$ in $[k]$, at most half the arcs entering $u$ have the color $\alpha$ and at most half of the arcs leaving $u$ have the color $\alpha$. A symmetric digraph $\overleftrightarrow{G}$ is obtained from an undirected graph~$G$ by replacing each edge $uv$ of $G$  by a pair of opposite arcs $\overrightarrow{uv}$ and $\overrightarrow{vu}$.

The paper is organized as follows. 
In Section~\ref{sec:notation} we introduce the used notations. 
In Section~\ref{sec:preliminaries} we state preliminary results useful to further theorems.
In Section~\ref{sec:majority} we show the examples of families of graphs with majority distinguishing 2-edge coloring. We also consider the majority distinguishing index for traceable graphs, especially complete graphs and complete balanced bipartite graphs. 
In Section~\ref{sec:main} we introduce the main result, which is the bound for majority distinguishing index, with respect to $\Delta$, for any graph (Theorem~\ref{thm:main}). 
Finally, in Section~\ref{sec:digraphs} we state a result for the majority index for any bipartite graph and we consider the  majority and majority distinguishing coloring for symmetric digraphs and we show the bound on the majority distinguishing index.  

\section{Notation}\label{sec:notation}

Majority edge coloring is well-defined only for graphs with minimum degree at least two. For simplicity, in the proofs, we extend this definition to all graphs. For a graph $G$ we define an {\it almost majority edge coloring} to be an edge coloring of $G$ such that for every vertex of degree at least two at most half of the edges incident with it have the same color and the edge incident with a vertex of degree one has any color. In particular, a graph $G$ with $\delta(G) \geq 2$ has an almost majority edge coloring with $k$ colors if it has a majority edge coloring with $k$ colors. 
For brevity, for any color~$\alpha$ and any vertex $u$, we denote by $d^{\alpha}(u)$ the number of edges in color~$\alpha$ incident with $u$.
Moreover, we define a {\it weak majority coloring} of a graph $G$ as an edge coloring $c \colon E(G) \rightarrow [k]$ such that for every vertex $u$ of $G$ and every color $\alpha$ in $[k]$ it holds $d^{\alpha}(u) \leq \ceil{\frac{d(u)}{2}}$.

An {\it asymmetric spanning subgraph} $H$ of $G$ is a subgraph of graph $G$ such that $V(H)=V(G)$ and $H$ has only the identity automorphism. For two graphs $G$ and~$H$, such that $G \subset H$, we define $H-G$ as a graph with the same set of vertices as~$H$ and the edge set $E(H) \setminus E(G)$.

For a given vertex $a$ of a graph $H$, we denote by $\Aut(H)_a$ the stabilizer of a~vertex $a$, i.e. $\Aut(H)_a=\{\varphi \in \Aut(H) : \varphi(a)=a\}$. For two vertices $a,b$, we denote $\Aut(H)_{a,b}=\Aut(H)_a \cap \Aut(H)_b$.
By $\dist(a,b)$ we denote the length of the shortest path between vertices $a$ and $b$.
For $r \in \{0,1, \ldots, \dist(a,b)\}$, let
$$S_r(a)=\{v \in V(H) : \dist (a,v)=r\}$$ be the $r$-th sphere at the vertex $a$. 

Two edge colorings $c_1, c_2$ of a graph $H$ we call \textit{isomorphic with respect to a group} $\Gamma \subseteq \Aut(H)$ if there exists an automorphism $\varphi \in \Gamma$ such that $c_1(uv)=c_2(\varphi(u)\varphi(v))$ for every edge $uv \in E(H)$. Furthermore, $c_1, c_2$ are called \textit{isomorphic} if they are isomorphic with respect to $\Aut(H)$. 

\section{Preliminaries}\label{sec:preliminaries}

In order to introduce the general results bounding the majority distinguishing index, we state the following corollary and lemma. 

\begin{corollary}\label{cor:almost}
Every graph $G$ 
has an almost majority coloring with 4 colors.
\end{corollary}

Above corollary is a reformulation of Theorem~\ref{thm:deg2} in context of the definition of almost majority coloring. Let us notice that this corollary can be strengthened in some special cases. If 
the degrees of all vertices of $G$ except for leaves are even, then $G$ has an almost majority coloring with two colors.  

\begin{lemma}\label{lem:a+b}
    Let $G$ be a graph with $\delta(G)\geq2$ and $H$ be its connected spanning subgraph with an $a$-coloring such that for every vertex $v \in V(H)$ and color $\alpha \in [a]$ it holds $d^{\alpha}_H(v) \leq \frac{d_G(v)}{2}$. If $b$ is the number of colors in weak majority coloring of $G-H$, then $M'_D(G) \leq a+b$.
\end{lemma}

\begin{proof}
    If $\alpha$ is a color of $H$, then for any vertex $v$ it holds $d_G^{\alpha}(v)=d_H^{\alpha}(v) \leq \frac{d_G(v)}{2}$. 
    Otherwise, for any vertex $v$ it holds $d_G^{\alpha}(v)=d_{G-H}^{\alpha}(v) \leq \ceil{\frac{d_{G-H}(v)}{2}} \leq \frac{d_G(v)}{2}$.
\end{proof}

\begin{lemma}\label{lem:asymmetric}
    Let $H$ be a connected asymmetric spanning subgraph of a graph $G$ with $\delta(G) \geq 2$. If $H$ has an almost majority coloring with $a$ colors, then $M'_D(G) \leq a+3$. 
\end{lemma} 

\begin{proof}
    We take an almost majority coloring of $H$ with $a$ colors and weak majority coloring of $G-H$ from Theorem~\ref{thm:2colors}. 
    From Lemma~\ref{lem:a+b} this coloring is majority edge coloring.
    If $\varphi$ is an automorphism of $H$ preserving this coloring, then $\varphi(x)=x$, for each $x \in V (H)$. Since $H$ is a spanning subgraph of $G$, therefore, this coloring is distinguishing. 
\end{proof}

We now introduce general results giving bounds on the majority distinguishing index.

\begin{proposition}\label{prop:deg2MD7}
    Every graph of minimum degree at least 2, which has a connected asymmetric spanning subgraph, has a majority distinguishing edge 7-coloring.
\end{proposition}

\begin{proof}    
     Let $H$ be a connected asymmetric spanning subgraph of $G$. By Corollary~\ref{cor:almost} there exists an almost majority edge coloring of $H$ with four colors. The result follows directly from  Lemma~\ref{lem:asymmetric}. 
\end{proof}

If the asymmetric subgraph is of Class 1, we need fewer colors. From the proof of Proposition~\ref{prop:deg2MD7} and Theorem~\ref{thm:deg4} we make a simple corollary.

\begin{corollary}
    Every graph of minimum degree at least 4, which has a connected asymmetric spanning subgraph of minimum degree at least 4, has a majority distinguishing edge 6-coloring.
\end{corollary}

We believe that the bound in Proposition~\ref{prop:deg2MD7} is not sharp. Therefore, we state the following conjecture.

\begin{conjecture}
    Every graph of minimum degree at least 2, which has a connected asymmetric spanning subgraph, has a majority distinguishing edge 5-coloring.
\end{conjecture}

\section{Majority distinguishing edge coloring}\label{sec:majority}
We start this section by showing that there exist graphs with a majority distinguishing $2$-edge coloring. Obviously, such graphs are Eulerian with an even number of edges. We observe that every graph with all vertices of even degree and the identity automorphism group has majority distinguishing 2-edge coloring. Examples of families of such graphs are listed below.

\begin{itemize}
    \item Let $G_k$ be a graph constructed from a cycle $C_{2k}$ with vertex set $\{v_1, v_2, \ldots, v_{2k}\}$ by adding a path of length $l_i$ between the vertices $v_{i}$ and $v_{i+1}$, for $i \in \{1,2, \ldots , 2k\}$. We require that the sum of lengths of added paths is even and that the vertex coloring of $C_{2k}$ obtained by taking the lengths of added paths as the colors is distinguishing. Note that the degrees of vertices of $G_k$ are in the set $\{2,4\}$, that $G_k$ is traceable and asymmetric. 
    \item Let $G'_k$ be a graph constructed in similar way as previously, with one exception. Namely, we add $k$ paths, the $i$-th path is added between the vertices $v_{2i}$ and $v_{2i+2}$. The graph $G'_k$ is also asymmetric but no longer traceable.
    \item Let $C$ be a graph obtained by choosing an odd number $i \geq 3$ of cycles from the set $\{C_3, C_4, C_4, C_5, C_6, C_6, C_7, C_8, C_8, \ldots \}$ and gluing them together in one edge of each cycle. We call this edge the central edge of $C$. This graph has an Eulerian trail. If we have the odd number of cycles of odd length, we subdivide the central edge. We color all edges alternately with two colors starting from the shortest cycle and continuing with the increasing length. (Of course we go only once through the central edge.) This technique of coloring ensures that we distinguish cycles of the same length. Note that the graph $C$ is traceable if and only if it has even number of cycles of odd length and its automorphism group is arbitrary.
    \item Let $C'$ be a graph obtained from the set of distinct cycles of even length by choosing any number of cycles and gluing them together in one vertex. It is a traceable graph with arbitrary automorphism group.
\end{itemize}


Presented families of graphs are not the only ones with a majority distinguishing 2-edge coloring. We find it hard to characterize all graphs with such property. 

\vspace{0.5cm}



We consider basic classes of graphs and show sharp upper bounds on the majority distinguishing index.

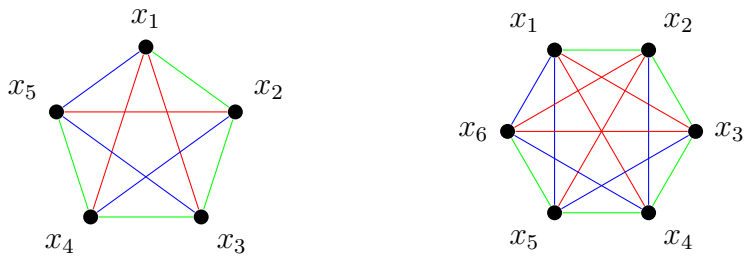
\begin{figure}[ht]
\hspace{1cm}
\begin{minipage}{.4\textwidth}
\begin{tikzpicture}[scale=.25] 
\def\radius{5}
\def\bender{0}

  \foreach \v/\lab in {1/{$x_1$}, 2/{$x_2$}, 3/{$x_3$}, 4/{$x_4$}, 5/{$x_5$}}{
    \node[circle,fill,inner sep=2pt, label={162-\v*72:\lab}] (n\v) at (162-\v*72:\radius cm) {};
  } 
  \foreach \v/\vnext/\col in {1/2/green, 2/3/green, 3/4/green, 4/5/green, 1/3/red, 1/4/red, 1/5/blue, 2/4/blue, 2/5/red, 3/5/blue}{
    \draw [\col] (n\v) to[bend right=\bender] (n\vnext); 
  }
\end{tikzpicture}
\centering
\end{minipage}
\hspace{1cm}
\begin{minipage}{.35\textwidth}
\begin{tikzpicture}[scale=.25] 
\def\radius{5}
\def\bender{0}

  \foreach \v/\lab in {1/{$x_1$}, 2/{$x_2$}, 3/{$x_3$}, 4/{$x_4$}, 5/{$x_5$}, 6/{$x_6$}}{
    \node[circle,fill,inner sep=2pt, label={180-\v*60:\lab}] (n\v) at (180-\v*60:\radius cm) {};
  } 
  \foreach \v/\vnext/\col in {1/6/blue, 1/2/green, 2/3/green, 3/4/green, 4/5/green, 5/6/green, 1/3/red, 1/4/red, 1/5/blue, 2/4/blue, 2/5/red, 2/6/red, 3/5/blue, 3/6/red, 4/6/blue}{
    \draw [\col] (n\v) to[bend right=\bender] (n\vnext); 
  }
\end{tikzpicture}  
\end{minipage}
\caption{Majority distinguishing coloring of $K_5$ and $K_6$.}\label{fig:K_5,K_6}
\end{figure}

\begin{proposition}\label{prop:complete} 
    For $n \geq 3$ we have $$M'_D(K_{n}) \leq 3,$$ except for $n=4$, where $M'_D(K_{4}) = 5$.
\end{proposition}

\begin{proof}
    For $n \geq 7$ we take an asymmetric spanning tree $T$ of $K_{n}$ with $\Delta(T)=3$ and color it with green. By Lemma~\ref{lem:asymmetric} there exists a majority distinguishing coloring with three colors.


    For $n \in \{3,4\}$, the majority distinguishing edge coloring is proper distinguishing coloring, hence from Theorem~\ref{thm:KP15} it holds $M'_D(K_{3}) = 3$ and $M'_D(K_{4}) = 5.$

    For $n \in \{5,6\}$ we take a spanning path $P=x_1x_2 \ldots x_n$ of $K_{n}$ and color it with green. We distinguish the ends of $P$ (final vertices: $x_1$ and $x_n$ or penultimates: $x_2$ and $x_{n-1}$) by the following coloring of the edges incident with them. In $K_{5}$ we color red the edges $x_1x_3,x_1x_4,x_2x_5$ and blue $x_1x_5, x_2x_4$, $x_3x_5$. In $K_{6}$ we color red the edges $x_1x_3,x_1x_4,x_2x_5,x_2x_6$, $x_3x_6$ and blue $x_1x_5, x_1x_6, x_2x_4$, $x_3x_5$, $x_4x_6$ (see Figure~\ref{fig:K_5,K_6}). This coloring is majority because for every vertex at most two incident edges have the same color. 
\end{proof}

From Theorem~\ref{thm:2colors} the bound from Proposition~\ref{prop:complete} is optimal for even $n$ and $n=4k+3$. The only possibility for $M'_D(K_{4k+1})=2$, where $k \in \mathbb{N}$, is if there exists an asymmetric $2k$-regular graph $H$ with $4k+1$ vertices. By~\cite{KSV} for $k$ big enough these graphs exist.

\begin{proposition}\label{prop:traceabe}
    If $G$ is a traceable graph and $\delta(G) \geq 4$, then $M'_D(G) \leq 3$.
\end{proposition}

\begin{proof}
    Let $P_n=\{x_1, \ldots, x_n\}$ be the spanning path of the graph $G$. We consider two cases based on whether $P_n$ has non-trivial automorphism in $G$. We color the edges of $P_n$ green, and next we color the remaining edges of $G$ with two colors: red and blue.

    First, assume that there exists an automorphism of the path $P_n$ in $G$. Then $d_G(x_1) = d_G(x_n)$ and for any vertex $v_i \in V(G)$, there exists $x_1v_i \in E(G)$ if and only if $v_{i'}x_n \in E(G)$, where $i'=n-i+1$. We color $x_1v_i$ red and $v_{i'}x_n$ blue for the smallest $i>2$. 
    If $d_G(x_1)$ is odd, then for $j \neq i'$, $j<n$ and $x_1v_j$ we color also $x_1v_j$ blue and $v_{j'}x_n$ red. Now, there does not exists any automorphism of G exchanging $v_1$ with $v_n$. By Theorem~\ref{thm:2colors} we color the edges of every connected component of $G - P_n$ with red and blue, such that the special vertex is different from $x_1$ and $x_n$. This can always be done because $\delta(G) \geq 4$.  
    
    This coloring is a majority coloring, because for $x_1$ and $x_n$ at most half of the incident edges have the same color, and for every other vertex $v \in V(G)$ at most $\ceil{\frac{d_{G-P_n}(v)+2}{2}} \leq \frac{d_{G}(v)}{2}$ of the edges incident with $v$ have the same color. Since the color of $x_1v_i$ is different from the color of $x_n v_{i'}$, then the only automorphism of $G$ that preserves this coloring is the identity automorphism. Therefore, this coloring is distinguishing.

    Otherwise, the path is already fixed and by Theorem~\ref{thm:2colors} we color the remaining edges with red and blue as previously.
\end{proof}

\begin{corollary}
    For $n \geq 4$ we have $$M'_D(K_{n,n}) \leq 3.$$ 
\end{corollary}

The proof is an immediate conclusion from Proposition~\ref{prop:traceabe}.




\vspace{0.5cm}
The difference between the number of colors needed in majority and majority distinguishing coloring can be arbitrary large. For example, $M'(K_{2,n}) \leq 3$ and Proposition~\ref{prop:K2n} shows that $M'_D(K_{2,n}) \approx \sqrt{n}$.

\begin{proposition}\label{prop:K2n}
    For $n\geq3$ we have $$M'_D(K_{2,n}) \leq \min\{k \in {\mathbb{N}}: k(k-1)-1 \geq n\}.$$
\end{proposition}

\begin{proof}
    Let $|X|=2$ and $|Y|=n$. In the beginning we assume that the vertices of~$X$ are fixed and we choose any order of them. For every vertex in $v \in Y$ we assign a pair which corresponds to the colors of the edges from $v$ to consecutive vertices in $X$. The coloring we wish to obtain is a majority coloring, hence each color can be used at most once in any pair. Therefore, there are $k(k-1)$ such pairs, where $k$ is the number of colors. We distinguished the vertices in~$Y$.

    In the beginning we assumed that the vertices of $X$ are fixed. We make sure that is the case by assigning them different sets of colors, which are colors on their incident edges. The only thing we can do is choosing the sequences for $Y$. Notice, that if $n \geq 3$, the $k$ is always at least $3$. Additionally, if we use all possible pairs for the vertices of $Y$, then both vertices in $X$ will have the same sets consisting of all the colors appearing $\frac{n}{k}$ times.
    Therefore, to distinguish the vertices in $X$, we remove the pair $(1,2)$ from the possible pairs assigned to the vertices of $Y$ and we always use the pair $(2,1)$ the vertices in $X$ we distinguish by missing colors. 

    Choosing the sequences for vertices in $Y$ we have to do it in a balanced way, in order to have the majority coloring in the vertices of $X$. 

\end{proof} 

\section{Main result}\label{sec:main}

\begin{theorem}\label{thm:main}
    If $G$ is a connected graph with minimum degree $\delta \geq 2$, then $$ M'_D(G) \leq \ceil{\sqrt{\Delta(G)}}+5. $$
\end{theorem}

In order to prove main Theorem, at the beginning we will prove it for 2-connected graphs and then for the graphs of connectivity 1. We follow the idea of the proof presented in \cite{IKPW}. We use similar lemmas, stating them here for completeness. However, we omit the details of the proofs and only make clear the changes that were necessary for the resulting coloring to be a majority coloring.

In the proof of the main result, we construct a majority distinguishing coloring, using colors from the set $Z=\{0,0',1,\ldots,\ceil{\sqrt{\Delta}}+3\}$. As noted earlier, we may assume that $\Delta(G) \geq 4$. Observe that we always have at least seven colors at our disposal.

\begin{lemma}\label{lem:H}
Let $a, b$ be two vertices of a graph $H$ of maximum degree at most $\Delta$, such that 
$$\dist(a,v)+\dist(v,b)=\dist(a,b)$$
for every vertex $v \in V(H)$. Then $H$ admits an almost majority coloring with $\ceil{\sqrt{\Delta}}+3$ colors breaking every automorphism of $\Aut(H)_{a,b}$.
\end{lemma}

\begin{proof}
For $r<\dist(a,b)$, every vertex $v \in S_r(a)$ has at least one and at most $\Delta-1$ neighbors in $S_{r+1}(a)$. For $r \geq 1$, denote $H_r=H[S_0(a)\cup\ldots \cup S_r(a)]$. We recursively color the edges between $S_r(a)$ and $S_{r+1}(a)$ with $\ceil{\sqrt{\Delta}}+3$ colors such that for each $r$ the following four conditions are satisfied:
\begin{itemize}
\item the coloring of $H_r$ is an almost majority coloring;
\item the coloring of $H_{r+1}$ is an weak majority coloring;
\item if $S_{r+1}(a)$ is fixed pointwise by every automorphism $\varphi \in \Aut(H)_{a,b}$ preserving the coloring of $H_{r+1}$, then $S_r(a)$ is also fixed pointwise by $\varphi$;
\item if $A \subseteq S_{r+1}(a)$ is a set of vertices such that there exists a cyclic permutation of $A$ that can be extended to an automorphism $\varphi \in \Aut(H)_{a,b}$ preserving the coloring of $H_{r+1}$, then $|A|\leq \ceil{\sqrt{\Delta}}$.
\end{itemize}

The first step of our coloring is identical to that presented in \cite{IKPW}. 
Namely, we partition the set of edges incident with $a$ into at most $\ceil{\sqrt{\Delta}}$ subsets of cardinality at most $\ceil{\sqrt{\Delta}}$, with an additional requirement that no subset has cardinality greater than half of $d(a)$. 
We color edges in each subset with the same color. 
If $d(a)=1$, then we color the edge incident with $a$ any color. 

Suppose that we have already defined an edge coloring $f$ of $H_r$ satisfying the above four conditions for $r-1$ instead of $r$. 
Let $A=\{v_1, \ldots, v_p\}$, with $p\geq2$, be a set of vertices of $S_r(a)$ such that there exists a cyclic permutation of $A$ that can be extended to an automorphism $\varphi \in \Aut(H)_{a,b}$ preserving the coloring $f$. 
By assumption, $1 \leq |A| \leq \ceil{\sqrt{\Delta}}$. 
Each vertex $v_i$ of $A$ has the same number $k\leq \Delta-1$ of incident edges joining it to $S_{r+1}(a)$.

Consider the vertex $v_i \in A$ for $i \in \{1, \ldots, p\}$. Let $N_{r+1}(v_i):=N(v_i) \cap S_{r+1}(a)$ and $N_{r-1}(v_i):=N(v_i) \cap S_{r-1}(a)$. We say that the color is \textit{forbidden} in vertex $v \in S_{r}(a)$ if it occurs $\ceil{\frac{d}{2}}$ times, the color is \textit{forbidden} in vertex $v \in S_{r+1}(a)$ if it occurs more than $\frac{d}{2}$ times, where $d$ is a number of already colored edges incident with $v$. 
Except for the case $k=1$ we have $\ceil{\sqrt{\Delta}}+1$ available colors and we ensure that for each vertex $v_i \in A$ we can assign distinct multiset of colors. 

If $k=1$, then for $v_i$ there are at most two forbidden colors from $N_{r-1}(v_i)$. Additionally we forbid one color which occurs the most often as a forbidden color among the vertices from $N_{r+1}(v_i)$. Therefore we have at least $\ceil{\sqrt{\Delta}}$ colors and in $A$ there are also at most $\ceil{\sqrt{\Delta}}$ vertices.
For every $v_i \in A$ we assign a distinct color. Notice that this coloring is almost majority coloring.

If $2 \leq k \leq \ceil{\sqrt{\Delta}}$, we forbid one from the forbidden colors from $N_{r-1}(v_i)$ and one color which occurs the most often as a forbidden color amount the vertices from $N_{r+1}(v_i)$. Therefore we have at least $\ceil{\sqrt{\Delta}}+1$ colors at out disposal. For every $v_i \in A$ we assign distinct set of $k$ colors. 

If $k \geq \ceil{\sqrt{\Delta}}+1$, we do not use one from the forbidden colors from $N_{r-1}(v_i)$ and one color which occurs the most often as a forbidden color amount the vertices from $N_{r+1}(v_i)$. Therefore we have at least $\ceil{\sqrt{\Delta}}+1$ colors. When $k \neq 4$, for every $v_i \in A$ we assign distinct color, which will occur only once. The remaining $k-1$ edges we color majority using each of the remaining colors at least twice.
For $k=4$, for every $v_i \in A$ we assign distinct pair of colors, which we use twice. 

Now, using the multisets of colors assigned to $v_i$, we color its incident edges. For any edge incident with $v_i$ we can assign a color which is not forbidden in any of its ends, because any color appears at most $\floor{\frac{k}{2}}$ times in multiset, it is not forbidden in at least $\ceil{\frac{k}{2}}$ in $N_{r+1}(v_i)$ and any vertex from $N_{r+1}(v_i)$ has at most one forbidden color.

Thanks to forbidden colors and the fact that in multisets assigned to each $v_i$ one color appears at most half times, this coloring is an almost majority coloring.

Furthermore, this way we produce at most $\ceil{\sqrt{\Delta}}$ vertices in $S_{r+1}(a)$ that can be interchanged by any automorphism preserving the coloring of $H_{r+1}$ because at most $\ceil{\sqrt{\Delta}}$ edges joining any vertex of $A$ to $S_{r+1}(a)$ have the same color.

The second last sphere $S_r(a)$, i.e. for $r=\dist(a,b)-1$, has at most $\Delta$ vertices and, due to our construction, any of its subsets $A$ that can be permuted, has at most $\ceil{\sqrt{\Delta}}$ vertices. If $d(b)=1$, then we color the edge incident with $b$ with any color. We color the edges between $b$ and the vertices of $A$, such that $|A| \geq 2$, with distinct colors, without the forbidden colors for $v_i$. Next we color the possibly remaining edges between $b$ and the vertices of $A$, such that $|A|=1$, with any colors, keeping the majority coloring in $b$. The unique vertex $b$ of the last sphere is fixed by assumption, hence all spheres are fixed pointwise with respect to any automorphism $\varphi \in \Aut(H)_{a,b}$ preserving the coloring $f$ of $H$. 

Finally, we color edges $xy$ within each sphere with an arbitrary color which is not forbidden neither in $x$ nor in $y$. It is possible, since we have at least 5 colors.
\end{proof}

Given a cycle $C$ of a 2-connected graph $G$ and two distinct colors $\alpha, \beta \in Z \setminus \{0, 0'\}$ by $C_0(\alpha,\beta)$ we denote this cycle colored such that three consecutive edges are colored with $\alpha, 0, \beta$ in that order, and all other edges of the cycle are colored with alternating $0'$ and $0$.

\begin{theorem}\label{thm:2-connected}
    If $G$ is a 2-connected, finite graph with maximum degree $\Delta$, then 
    $$M'_D(G) \leq \ceil{\sqrt{\Delta}}+5.$$
\end{theorem}

\begin{proof}
If in the graph $G$ the length of any cycle is at most four, then $\Delta(G)\leq3$ or $G=K_{2,\Delta}$ or $G=K_{2,\Delta}+e$, where $e$ is an edge between two vertices of maximum degree $\Delta$. 
Notice, that for $\Delta \leq 3$ by~\cite{KP} it holds that $M'_D(G)=\chi'_D(G) \leq \Delta+2 \leq \ceil{\sqrt{\Delta}}+5$.
From Proposition~\ref{prop:K2n}  we have that $M'_D(K_{2,n}) \leq \min \{k : k(k-1)-1 \geq \Delta \}<\ceil{\sqrt{\Delta}}+5$.
Hence, we will prove the Theorem for $\Delta \geq 4$.

We proceed similarly to the proof of Theorem~2.2 in~\cite{KP}. We take the longest cycle $C$ in $G$ and color it as $C_0(\alpha,\beta)$ for some distinct $\alpha, \beta \in Z \setminus \{0, 0'\}$. Moreover we color all chords of $C$ with some distinct $\gamma, \delta \in Z \setminus \{0, 0', \alpha, \beta\}$ in such a way that this coloring of cycle and chords is majority. We will never use the colors $0, 0'$ outside $C$. Thus each vertex of $C$ is fixed by every automorphism of $G$ preserving the coloring, because the endpoints of the edge with color $\beta$ are fixed.  

Then, we proceed identical like in the proof of Theorem~2.2 in~\cite{KP}. We color the remaining edges of $G$ recursively. Assuming we have already colored the edges of a 2-connected subgraph $F$ of $G$ and that there exists a vertex incident only with edges not yet colored, we take the shortest path with endpoints in $F$ and at least one vertex outside of $F$. We use this path to build a graph $H$, which we color from Lemma~\ref{lem:H} using at most $\ceil{\sqrt{\Delta}}+3$ colors. If $d_H(a)=1$ (or $d_H(b)=1$), we color the edge incident with $a$ (or $b$) with a color that was assigned to fewer than half of the colored edges incident with $a$ (or $b$). We apply this procedure until all vertices of $G$ have at least one incident edge colored. If there are uncolored edges, we color them in order to obtain a majority coloring. For each edge at most two colors are forbidden at each vertex. Since we have at least five colors (except for $0$ and $0'$), then we can always choose a color for an edge.
\end{proof}

Actually, we have proved the following result which we use for the graphs of connectivity one. 

\begin{corollary}\label{cor:2-connected}
    Let $G$ be a 2-connected graph with maximum degree $\Delta$ and let $C$ be the longest cycle or any cycle of length at least five in $G$. Then for every two distinct colors $\alpha, \beta \in Z \setminus \{0, 0'\}$, any coloring $C_0(\alpha, \beta)$ of its edges can be extended to a distinguishing coloring of $G$ with colors from the set $Z$ such that all edges colored with $0$ belong to the cycle $C$.
\end{corollary}

\begin{proof}
    The claim follows from the proof of Theorem~\ref{thm:2-connected}.
\end{proof}

Observe, that for $\Delta \leq 3$ by~\cite{KP} it holds $M'_D(G)=\chi'_D(G) \leq \Delta+2 \leq \ceil{\sqrt{\Delta}}+5$. Hence, from now on, we assume that $G$ is a graph of maximum degree $\Delta \geq 4$. The proofs of following lemmas follow straightforwardly from proofs of Lemma~3.1, Lemma~3.2 and Lemma~3.3 from~\cite{IKPW}. Since in general we have more colors, the calculations differ, hence we can lower necessary condition for $\Delta \geq 4$.

\begin{lemma}\label{lem:2-connected}
    Let $u_0$ be the vertex of a 2-connected block $H_0$ of maximum degree at most $\Delta$, where $\Delta \geq 4$. Then $H_0$ admits at least $\ceil{\sqrt{\Delta}}$ edge colorings with colors from the set $Z$, breaking all non-trivial automorphisms of $\Aut(H_0)_{u_0}$, which are pairwise non-isomorphic with respect to $\Aut(H_0)_{u_0}$, and do not contain $C_0(1,2)$.
\end{lemma}

\begin{lemma}\label{lem:1-connected}
    Let $H_0$ be a graph of maximal degree at most $\Delta$, with $\Delta \geq 4$, consisting of $s \geq 2$ copies of a 2-connected block $B$ sharing a common cut vertex $u_0$. Then $H_0$ admits at least $\ceil{\sqrt{\Delta}}$ pairwise non-isomorphic distinguishing edge colorings with colors from the set $Z$, such that $C_0(1,2)$ does not appear in any of these colorings.
\end{lemma}

\begin{lemma}\label{lem:lem3.3}
    Let $H_0$ be a graph satisfying the assumptions of Lemma~\ref{lem:2-connected} or of Lemma~\ref{lem:1-connected}. Let $T$ be a symmetric tree of order at least 3 with a central vertex $v_0$. A graph $H$ is obtained by attaching a copy of the graph $H_0$ to every leaf of $T$ in such a way that each pendant edge of $T$ is incident with the same vertex $u_0$ of $H_0$. If the maximum degree of $H$ is at most $\Delta$ and $\Delta \geq 4$, then there exists a distinguishing coloring of $H$ with colors from the set $Z$, and without $C_0(1,2)$.
\end{lemma}

Using the previous lemmas we prove the following theorem for graphs of connectivity one.

\begin{theorem}\label{thm:1-connected}
    Let $G$ be a graph of connectivity 1 and without pendant edges. If $G$ has a maximum degree $\Delta$, then 
    $$M'_D(G) \leq \ceil{\sqrt{\Delta}}+5. $$
\end{theorem}

\begin{proof}
    Again we use the idea of the proof of Theorem~3.5 from~\cite{IKPW}. 
    We choose any block $B_0$ from graph $G$ and color it like in Corollary~\ref{cor:2-connected}. Next we consider the cut vertices of $G$ in any fixed ordering $v_1, v_2, \ldots, v_p$, such that $v_i$ for $i \in [p]$ belong to colored block and uncolored block at the same time. We consider the following two stages.

    \begin{enumerate}
        \item Let $B_1, \ldots, B_l$ be pairwise non-isomorphic, uncolored 2-connected block containing $v_i$. For each $j \in \{1, \ldots, l\}$, we consider a maximal subgraph $H_0$ of $G$ that consists of $s \geq 1$ copies of $B_j$ sharing the cut vertex $v_i$. Taking $u_0=v_i$, by Lemma~\ref{lem:2-connected} if $s=1$, or by Lemma~\ref{lem:1-connected} if $s \geq 2$, there is a distinguishing coloring of $H_0$ with colors from the set $Z$, not containing $C_0(1,2)$.
        \item We consider every uncolored maximal subgraph $H$ satisfying the assumptions of Lemma~\ref{lem:lem3.3}. We color $H$ distinguishingly according to the conclusion of Lemma~\ref{lem:lem3.3}. Consequently, all 2-connected blocks containing $v_i$ are now colored.
    \end{enumerate}
    
    Finally, we color every yet uncolored edge $e$ incident with $v_i$ with any color which is not forbidden in its endpoints. Note that both endpoints of $e$ are already fixed by any $\varphi \in \Aut(G)$ preserving the existing partial coloring, because uncolored components of $G-v_i$ are pairwise non-isomorphic.

    This procedure yields a majority distinguishing coloring of $G$ with $\ceil{\sqrt{\Delta}}+5$ colors, which completes the proof of Theorem~\ref{thm:1-connected}, and thus Theorem~\ref{thm:main}.
\end{proof}

\section{Symmetric digraphs}\label{sec:digraphs}

In this section we consider majority coloring and majority distinguishing coloring for symmetric digraphs.
We start with a simple corollary from the proof of Theorem~\ref{thm:deg2} contained in~\cite{BKPPRW}, which is helpful to prove the bound on the majority index for digraphs. The key component of that proof is a decomposition of the edge set of a graph such that each vertex in a newly obtained graph is of degree two or three. The proper edge coloring of the new graph is used to obtain a majority coloring of the original graph. The fact that bipartite graphs are of Class 1 allows us to use three colors for the proper coloring of the new graph.

\begin{corollary}\label{cor:bipartite}
    Every bipartite graph of minimum degree at least 2 has a majority 3-edge-coloring.
\end{corollary}

\begin{theorem}
    If $G$ is a graph of minimum degree at least 2, then $M'(\overleftrightarrow{G})\leq 3$.
\end{theorem}

\begin{proof}
    We define an auxiliary bipartite graph $G'$ from the digraph $\overleftrightarrow{G}$ as follows.
    The set of vertices $V(G')$ is obtained from two copies of the set $V(\overleftrightarrow{G})$. We denote them $V_1$ and $V_2$. If $\overrightarrow{uv}$ is an arc of $\overleftrightarrow{G}$, then $u_1v_2$ in an edge of $G'$.
    
    Notice that for any $v_1 \in V_1$ it holds that $d_{G'}(v_1)=d^+_{\overleftrightarrow{G}}(v)$ and for any $v_2 \in V_2$ it holds that $d_{G'}(v_2)=d^-_{\overleftrightarrow{G}}(v)$. Moreover, there is a one-to-one correspondence between the edges of $G'$ and the arcs of $\overleftrightarrow{G}$.
    
    From Corollary~\ref{cor:bipartite} there exists a majority coloring of $G'$ with three colors. Hence, we obtain a coloring of the arcs of $\overleftrightarrow{G}$ with three colors which is a majority arc coloring.
\end{proof}

    We finish this section with a result for symmetric digraphs. 

\begin{theorem}\label{thm:digraphs}
    If $\overleftrightarrow{G}$ is a connected symmetric digraph with minimum degree $\delta(G) \geq 2$, then $$ M'_D(\overleftrightarrow{G}) \leq \ceil{\sqrt[4]{\Delta(G)}}+4. $$
\end{theorem}

\begin{proof}
    The proof is analogous to the proof of main Theorem~\ref{thm:main}. It differs in a few details.
    We construct a majority distinguishing coloring of a symmetric digraph~$\overleftrightarrow{G}$ using colors from the set $Z=\{0,1,\ldots,\ceil{\sqrt[4]{\Delta(G)}}+3\}$. 
    We color the symmetric directed cycle~$\overleftrightarrow{C_0}$ with $(0,\alpha,\beta) \in Z$ using color $0$ on one of the directed cycles of~$C_0$, and $\alpha$ and $\beta$ on the remaining arcs, but we use color $\alpha$ on exactly one arc. This way we have exactly one additional color $0$ for the cycle. 
    
    We proceed recursively in a similar way as in the proof for graphs, coloring the arcs between spheres $S_{r}$ and $S_{r+1}$ in each step. Consider the set $A$ of vertices that can be permuted to obtain an automorphism of $H$. The difference is that instead of assigning a color to an edge $uv$, we assign a pair of colors to the pair of arcs between the vertices $u$ and $v$. Hence, to obtain $\ceil{\sqrt{\Delta}}$ different pairs it suffices to use $\ceil{\sqrt[4]{\Delta}}$ different colors on each of the positions in the pair.
        
    Therefore it is enough to use $\ceil{\sqrt[4]{\Delta}}+4$ colors.
\end{proof}


\end{document}